\documentclass[11pt]{article}
\usepackage[latin9]{inputenc}
\usepackage{amsmath}
\usepackage{amssymb}
\usepackage{graphicx}

\makeatletter

\providecommand{\tabularnewline}{\\}

\@ifundefined{date}{}{\date{}}

\usepackage{amsthm}\usepackage{enumerate}\usepackage{fullpage}

\newcommand{\cost}{\text{cost}}
\newcommand{\dist}{\text{dist}}
\newcommand{\pr}[1]{\mathbb{P}\left[ #1 \right]}

\newtheorem{theorem}{Theorem}[section]
\newtheorem{lemma}{Lemma}[section]
\newtheorem{corollary}{Corollary}[section]

\makeatother

\begin{document}

\title{Anarchy is free in network creation}

\author{Ronald Graham%
\thanks{Department of Computer Science and Engineering, University of California,
San Diego, La Jolla, CA 92093. E-mail: \texttt{graham@ucsd.edu}.%
} \and Linus Hamilton%
\thanks{Department of Mathematical Sciences, Carnegie Mellon University, Pittsburgh,
PA 15213. E-mail: \texttt{luh@andrew.cmu.edu}. Research supported
by NSF grant DMS-1041500.%
} \and Ariel Levavi%
\thanks{Department of Computer Science and Engineering, University of California,
San Diego, La Jolla, CA 92093. E-mail: \texttt{alevavi@cs.ucsd.edu}.%
} \and Po-Shen Loh%
\thanks{Department of Mathematical Sciences, Carnegie Mellon University, Pittsburgh,
PA 15213. E-mail: \texttt{ploh@cmu.edu}. Research supported by NSF
grants DMS-1201380 and DMS-1041500, an NSA Young Investigators Grant
and a USA-Israel BSF Grant.%
} }
\maketitle
\begin{abstract}
The Internet has emerged as perhaps the most important network in
modern computing, but rather miraculously, it was created through
the individual actions of a multitude of agents rather than by a central
planning authority. This motivates the game theoretic study of network
formation, and our paper considers one of the most-well studied models,
originally proposed by Fabrikant et al. In it, each of $n$ agents
corresponds to a vertex, which can create edges to other vertices
at a cost of $\alpha$ each, for some parameter $\alpha$. Every edge
can be freely used by every vertex, regardless of who paid the creation
cost. To reflect the desire to be close to other vertices, each agent's
cost function is further augmented by the sum total of all (graph
theoretic) distances to all other vertices.

Previous research proved that for many regimes of the $(\alpha,n)$
parameter space, the total social cost (sum of all agents' costs)
of every Nash equilibrium is bounded by at most a constant multiple
of the optimal social cost. In algorithmic game theoretic nomenclature,
this approximation ratio is called the price of anarchy. In our paper,
we significantly sharpen some of those results, proving that for all
constant non-integral $\alpha>2$, the price of anarchy is in fact
$1+o(1)$, i.e., not only is it bounded by a constant, but it tends
to 1 as $n\rightarrow\infty$. For constant integral $\alpha\geq2$,
we show that the price of anarchy is bounded away from 1. We provide
quantitative estimates on the rates of convergence for both results. 
\end{abstract}

\section{Introduction}

Networks are of fundamental importance in modern computing, and substantial
research has been invested in network design and optimization. However,
one of the most significant networks, the Internet, was not created
``top-down'' by a central planning authority. Instead, it was constructed
through the cumulative actions of countless agents, many of whom built
connections to optimize their individual objectives. To understand
the dynamics of the resulting system, and to answer the important
question of how much inefficiency is introduced through the selfish
actions of the agents, it is therefore natural to study it through
the lens of game theory.

In this paper, we focus on a well-studied game-theoretic model of
network creation, which was formulated by Fabrikant et al. in \cite{FaLuMaPaSh}.
There are $n$ agents, each corresponding to a vertex. They form a
network by laying down connections (edges) between pairs of vertices.
For this, each agent $v$ has an individual strategy, which consists
of a subset $S_{v}$ of the rest of the vertices that it will connect
to. The resulting network is the disjoint union of all (undirected)
edges between vertices $v$ and vertices in their $S_{v}$. Note that
in this formulation, an edge may appear twice, if $v$ lays a connection
to $w$ and $w$ lays a connection to $v$. Let $\alpha$ be an arbitrary
parameter, which represents the cost of making a connection. In order
to incorporate each agent's desire to be near other vertices, the
total cost to each agent is defined to be: 
\[
\cost(v)=\alpha|S_{v}|+\sum_{w}\dist(v,w)\,,
\]
where the sum is over all vertices in the graph, and $\dist(v,w)$
is the number of edges in the shortest path between $v$ and $w$
in the graph, or infinity if $v$ and $w$ are disconnected. The \emph{social
cost}\/ is defined as the total of the individual costs incurred
by each agent. This cost function summarizes the fact that $v$ must
pay the construction cost for the connections that it initiates, but
$v$ also prefers to be graph-theoretically close to the other nodes
in the network. This model also encapsulates the fact that, just as
in the Internet, once a connection is made, it can be shared by all
agents regardless of who paid the construction cost.

The application of approaches from algorithmic game theory to the
study of networks is not new.  The works \cite{AnDaKlTaWeRo,AnDaTaWe,CzVo,PaSTOC01,RoSTOC02,RoBOOK}
all consider network design issues such as load balancing,
routing, etc. Numerous papers, including
\cite{Al,AlEiEvMaRo,AlDeHaLe,AnFeMa,CoPa,DeHaMaZa} 
and the surveys \cite{Ja,TaWe},
have considered network formation itself,
by formulating and studying network creation games. From a game-theoretic perspective, a (pure) \emph{Nash
equilibrium} is a tuple of deterministic strategies $S_{v}$ (one
per agent) under which no individual agent can strictly reduce its
cost by unilaterally changing its strategy assuming all other agents
maintain their strategies. If every unilateral
deviation strictly increases the deviating agent's cost, then the
Nash equilibrium is \emph{strict}.

To quantify the cumulative losses incurred by the lack of coordination,
the key ratio is called the \emph{price of anarchy}, a term coined
by Koutsoupias and Papadimitriou \cite{KoPa}. It is defined as the
maximum social cost incurred by any Nash equilibrium, divided by the
minimum possible social cost incurred by any tuple of strategies.
Note that the minimizer, also known as the \emph{social optimum},
is not necessarily a Nash equilibrium itself. The central questions
in this area are thus to understand the price of anarchy, and to characterize
the Nash equilibria.

\subsection{Previous work}

To streamline our discussion, we will represent a tuple of strategies
with a directed graph, whose underlying undirected graph is the resulting
network, and where each edge $vw$ is oriented from $v$ to $w$ if
it was constructed by $v$'s strategy ($w\in S_{v}$). This is well-defined
because it is clear that the social optimum and all Nash equilibria
will avoid multiple edges, and so each edge is either not present
at all, or present with a single orientation.

The problem is trivial for $\alpha<1$, because all Nash equilibria
produce complete graphs, as does the social optimum, and therefore
the price of anarchy is 1 in this range. For $\alpha\geq1$, a new
Nash equilibrium arises: the star with all edges oriented away from
the central vertex. Indeed, the central vertex has no incentive to
disconnect any of the edges which it constructed, as its individual
cost function would rise to infinity, and no other vertex has incentive
to add more connections, because a new connection would cost an additional
$\alpha\geq1$, and reduce at most one of the pairwise distances by
1. Yet, as observed in the original paper of Fabrikant et al. \cite{FaLuMaPaSh},
when $\alpha<2$, the social optimum is a clique, and they calculate
the price of anarchy to be $\frac{4}{2+\alpha}+o(1)$, where the error
term tends to 0 as $n\rightarrow\infty$. This ranges from $\frac{4}{3}$
to 1 as $\alpha$ varies in that interval.

For $\alpha\geq2$, the social optimum is the star.  Various bounds on the
price of anarchy were achieved, with particular interest in constant
bounds, which were derived in many ranges of the parameter space. From the
point of view of approximation algorithms, these show that in those ranges
of $\alpha$, the Nash equilibria that arise from the framework of selfish
agents still are able to approximate the optimal social cost to within a
constant factor.  The current best bounds are summarized in Table
\ref{tab:prev-poa}.

\begin{table}[htbp]
\centering %
\begin{tabular}{l|l}
Regime  & Upper bound on price of anarchy \tabularnewline
\hline 
General $\alpha$  & $2^{O(\sqrt{\log n})}$ \tabularnewline
$2\leq\alpha<\sqrt[3]{n/2}$  & 4 \tabularnewline
$\sqrt[3]{n/2}\leq\alpha<\sqrt{n/2}$  & 6 \tabularnewline
$\alpha=O(n^{1-\epsilon})$  & $O(1)$ \tabularnewline
$\alpha\geq12n\lceil\log_{2}n\rceil$  & $O(1)$ \tabularnewline
\end{tabular}\caption{{\small{Previous upper bounds on the price of anarchy. The last bound
above is due to Albers et al. \cite{AlEiEvMaRo}, and the other bounds
are due to Demaine et al. \cite{DeHaMaZa}.}}}

{\small{\label{tab:prev-poa} }}
\end{table}

\subsection{Our contribution}

Much work had been done to achieve constant upper bounds on the price
of anarchy in various regimes of $\alpha$, because those imply the
satisfying conclusion that selfish agents fare at most a constant
factor worse than optimally coordinated agents. Perhaps surprisingly
(or perhaps reassuringly), it turns out that the price of anarchy
is actually $1+o(1)$ for most constant values of $\alpha$. In other
words, the lack of coordination has negligible effect on the social
cost as $n$ grows.

\begin{theorem} For non-integral $\alpha>2$, and $n > \alpha^3$, the
  price of anarchy is at most 
\begin{displaymath}
  1 + \frac{150\alpha^{6}}{(\alpha-\lfloor\alpha\rfloor)^{2}}
  \sqrt{\frac{\log n}{n}}  = 1+o(1) \,.
\end{displaymath}
On the other hand, for each integer $\alpha\geq2$, the price of anarchy
is at least 
\begin{displaymath}
  \frac{3}{2} - \frac{3}{4\alpha} + \frac{1}{\alpha^2} +o(1) \,,
\end{displaymath}
and it is achieved by the following construction. Start with an arbitrary
orientation of the complete graph on $k$ vertices. For each vertex
$v$ of the complete graph, add $\alpha-1$ new vertices, each with
a single edge oriented from $v$. 
\label{thm:main} 
\end{theorem}

\section{Proof for non-integral $\boldsymbol{\alpha}$}

Assume that we are given a Nash equilibrium.  In this section, we prove
that its total social cost is bounded by $1+o(1)$ times the social optimum,
as stated in Theorem \ref{thm:main}.  Throughout this proof, we impose a
structure on the graph as follows: select a vertex $v$, and partition the
remainder of the graph into sets based on their distance from $v$. Let
$N_{1}$ denote the set of vertices at distance 1 from $v$, let $N_{2}$
denote the set of vertices at distance 2 from $v$, etc., as diagrammed in
Figure \ref{fig:1}.  Since the graph in every Nash equilibrium is obviously
connected, every vertex falls into one of these sets.

\begin{figure}[htbp]
  \begin{center}
    \includegraphics[scale=0.8]{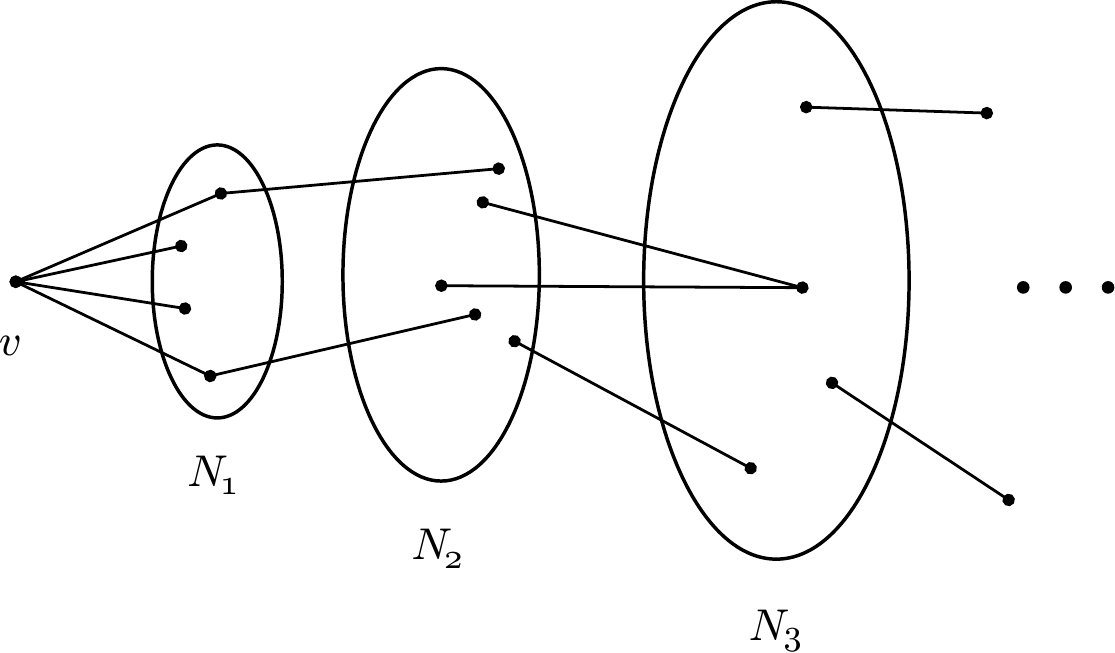}
  \end{center}
  \caption{\small Partitioning the graph into sets.}
  \label{fig:1}
\end{figure}

Consider any vertex $v_{i}\in N_{i}$ where $i\geq3$. Since the graph is
connected, we can always find a path $v_{i}v_{i-1}v_{i-2}\ldots
v_{2}v_{1}v$, where $v_{j}\in N_{j}$ for all $1\leq j\leq i$.  In
this case, we will call $v_{i}$ a \emph{child}\/ of $v_{2}$. (Note that
$v_{i}$ may be a child of more than one vertex, but is always a child of at
least one vertex.)  This is diagrammed in Figure \ref{fig:2}.

\begin{figure}[htbp]
  \begin{center}
    \includegraphics[scale=0.8]{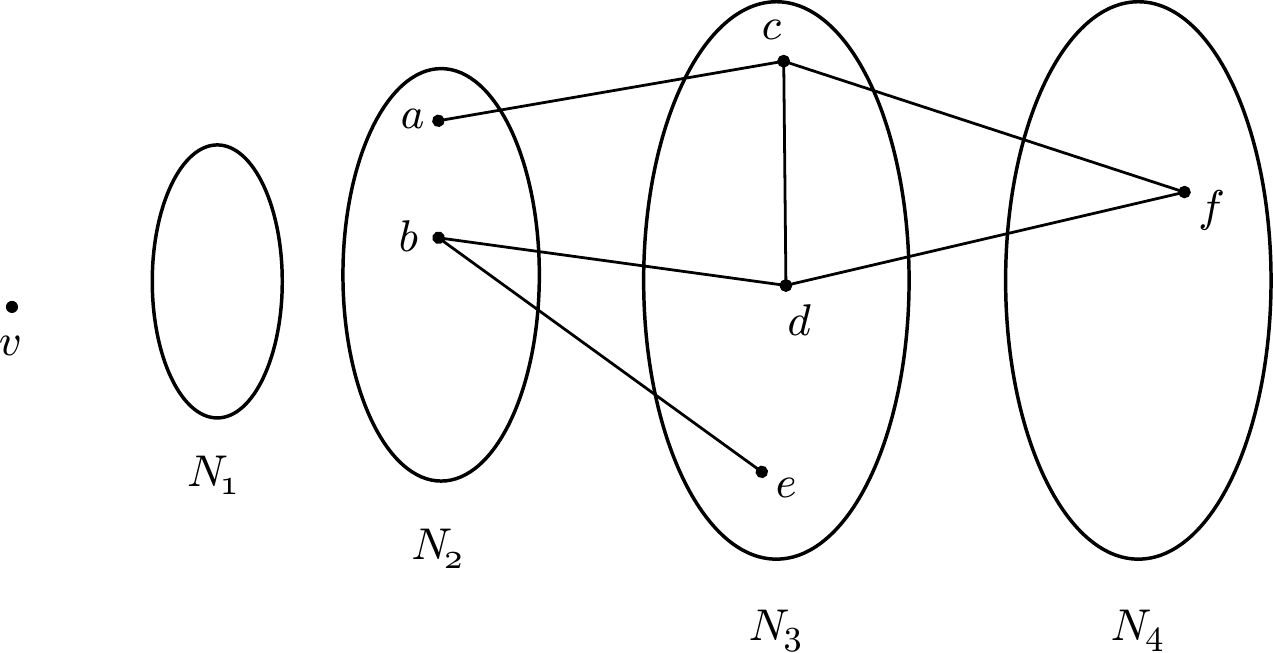}
  \end{center}
  \caption{\small Here, $d$ and $e$ are children of $b$; $c$ is a child of
  $a$, but not a child of $b$; and $f$ is a child of both $a$ and $b$.}
  \label{fig:2}
\end{figure}

\begin{lemma}
  No matter which vertex is used as $v$ to construct the vertex partition,
  every vertex in $N_{2}$ has at most $\lfloor \alpha-1 \rfloor$ children.
  \label{lem:N2-children}
\end{lemma}

\begin{proof}
  Suppose $w \in N_{2}$ has more than $\alpha-1$ children. Consider what
  happens if $v$ buys an edge to $w$.  Although $v$ pays $\alpha$ for the
  edge, it gets one step closer to $w$ and all of its children, and so the
  distance component of $v$'s cost function reduces by more than
  $1+(\alpha-1) = \alpha$. Therefore, buying the edge is a net positive
  gain for $v$.  But we assumed the graph was a Nash
  equilibrium---contradiction.  Therefore, $w$ has at most $\alpha-1$
  children, and since its number of children is an integer, we may round
  the bound down as in the statement of the lemma.
\end{proof}

\begin{lemma}
  Regardless of the choice of $v$, the resulting parts $N_i$ satisfy:
  \begin{displaymath}
    |N_{1}|+|N_{2}|+1
    \geq
    \frac{n}{\alpha} \,.
  \end{displaymath}
  \label{lem:N12-big}
\end{lemma}

\begin{proof}
  Since every vertex in $N_{3}\cup N_{4}\cup\ldots$ is a child of at least
  one vertex of $N_2$, but Lemma \ref{lem:N2-children} bounds the number of
  children per $N_2$-vertex by $\alpha-1$, we must have
  \begin{align*}
    (\alpha-1) |N_2|
    &\geq
    |N_{3}\cup N_{4}\cup\ldots|
    =
    (n-1-|N_{1}|-|N_{2}|) \\
    \alpha |N_2| + |N_1| + 1
    &\geq n
    \,,
  \end{align*}
  which implies the desired result.
\end{proof}

\begin{lemma}
  If $x$ has degree at least $\alpha$, then every vertex is at most
  distance 3 from it.
  \label{lem:big-dist-3}
\end{lemma}

\begin{proof}
  If some vertex $w$ is distance at least 4 from $x$, then $w$ can buy an
  edge to $x$.  Vertex $w$ will pay $\alpha$ for the edge, and get 3 steps
  closer to $x$, as well as at least 1 step closer to all of $x$'s
  immediate neighbors, for a net gain.  Hence this cannot appear in a Nash
  equilibrium.
\end{proof}

\begin{corollary}
  If $n$ is sufficiently large ($n>\alpha^{3})$, then the graph has
  diameter at most 4.
  \label{cor:diam-4}
\end{corollary}

\begin{proof}
  Consider an arbitrary pair of vertices $v,w$. Lemma \ref{lem:N12-big}
  implies that for $n$ sufficiently large ($n>\alpha^{3}$ suffices), either
  $v$ has degree at least $\alpha$, or one of $v$'s neighbors has degree at
  least $\alpha$.  In either case, we can travel from $v$ to a vertex with
  degree at least $\alpha$ in at most one step, and then by Lemma
  \ref{lem:big-dist-3}, travel to $w$ in at most 3 more steps. Therefore,
  $v$ and $w$ are at distance at most 4.
\end{proof}

\noindent \textbf{Remark.}  From now on, we will assume $n>\alpha^{3}$,
and so for any initial choice of $v$, the resulting partition will only
have $N_{1}$, $N_{2}$, $N_{3}$, and $N_{4}$.

\begin{lemma}
  Consider the partition constructed from an arbitrary initial vertex $v$.
  Select any $w \in N_{2}$, and let $d$ be the number of edges $w$ pays for
  which connect to other vertices in $N_{2}$.  Then
  $d\leq|N_{1}| \cdot \frac{\alpha}{\alpha-\lfloor\alpha\rfloor}$.
  \label{lem:N2-bdd-degree}
\end{lemma}

\begin{proof}
  Consider the following strategy for $w$: disconnect those $d$ edges, and
  instead connect to every vertex in $N_{1}$. We will carefully tally up
  the potential gain for this amendment.

  \begin{itemize}
    \item \emph{Paying for edges:} $w$ \textbf{saves at least
        $\boldsymbol{(d - |N_1|) \alpha}$} in terms of paying for edges.
        (The ``at least'' is because $w$ might already be connected to some
        vertices in $N_{1}$.)

    \item \emph{Connectedness to $v$ and $N_1$:} $w$ obviously can't get
      farther away from $v$ or any vertices in $N_{1}$.

    \item \emph{Connectedness within $N_2$:} $w$ gets farther away from all
      $d$ vertices it disconnected from, but remains at distance 2 from all
      of $N_{2}$, since every vertex in $N_{2}$ is connected to some vertex
      in $N_{1}$. This results in a maximum \textbf{increased cost of
        $\boldsymbol{d}$} in terms of distances to other vertices within
        $N_{2}$.

      \item \emph{Connectedness to $N_3$ and $N_4$:} When disconnecting
        from a vertex $x\in N_{2}$, $w$ might get farther away from all of
        $x$'s children in $N_{3}$ and $N_{4}$. However, remember that $w$
        is still distance 2 from all of $N_{2}$. Hence, $w$ is still
        distance 3 from all of $N_{3}$ and distance 4 from all of $N_{4}$.
        Therefore, $w$ can only get 1 step farther from $x$'s children, and
        doesn't get any farther from vertices in $N_{3}$ and $N_{4}$ that
        aren't $x$'s children.  By Lemma \ref{lem:N2-children}, every
        $N_2$-vertex has at most $\lfloor\alpha-1\rfloor$ children.
        Therefore, in disconnecting from $d$ vertices, $w$ gets 1 step
        farther from at most $d\lfloor\alpha-1\rfloor$ vertices in $N_{3}$
        and $N_{4}$, for a \textbf{cost increase of at most $\boldsymbol{d
        \lfloor \alpha - 1 \rfloor}$}. 
  \end{itemize}

  Adding, $w$'s net cost savings total to at least
  $(d-|N_{1}|)\alpha-d-d\lfloor\alpha-1\rfloor$, which must be $\leq 0$
  since we are at a Nash equilibrium.  Rearranging, $d \leq |N_{1}| \cdot
  \frac{\alpha}{\alpha-\lfloor\alpha\rfloor}$, as desired.
\end{proof}

\begin{lemma}
  If $|N_{1}|$ is $o(n)$, then so is $|N_{3}\cup N_{4}|$.  Quantitatively,
  $|N_{3}\cup
  N_{4}|<|N_{1}| \cdot \frac{5\alpha^{3}}{\alpha-\lfloor\alpha\rfloor}$.
  \label{lem:N34<N1}
\end{lemma}

\begin{proof}
  Let $P$ be the number of pairs of vertices $(x,y)$, such that $x \in
  N_{3}\cup N_{4}$ and $y$ is at most distance 2 from $x$. We will bound
  this number in two ways.  First, Lemma \ref{lem:N12-big} tells
  us that for any vertex in the graph, the number of vertices at most
  distance 2 from it is at least $\frac{n}{\alpha}$.  Therefore,
  $P\geq|N_{3}\cup N_{4}| \cdot \frac{n}{\alpha}$.

  For the second way, we will find an upper bound for the number of ways to
  start at a vertex $x \in N_{3}\cup N_{4}$, and then travel along at most
  two edges in some way. This is an overcount for $P$, so it will give an
  upper bound. To count the number of these paths, we do casework on the
  various ways to start at a vertex in $N_{3}\cup N_{4}$ and then travel
  along at most two edges.

  \medskip

  \emph{Case 1: The path stays inside $N_{3}\cup N_{4}$.} Any vertex in
  $N_{3}\cup N_{4}$ can be connected to at most $\alpha-1$ other vertices
  in $N_{3}\cup N_{4}$ (otherwise $v$ would gain from connecting to it
  directly), so the number of paths for us to count for each starting
  vertex is at most $1+(\alpha-1)+(\alpha-1)^{2}\leq\alpha^{2}$. Therefore,
  the total number of paths of this type is at most $\boldsymbol{|N_3 \cup
  N_4| \alpha^2}$.

  \medskip

  \emph{Case 2: The path travels from $N_{3}\cup N_{4}$ to $N_{3}$ to
  $N_{2}$, or is a length 1 path traveling from $N_{3}$ to $N_{2}$.} We
  count these backwards, starting from $N_2$.  The number of edges from
  $N_{2}$ to $N_{3}$ is at most $\alpha|N_{2}|$ by Lemma
  \ref{lem:N2-children}, and again, every vertex in $N_{3}$ is connected to
  at most $\alpha$ vertices in $|N_{3}\cup N_{4}|$, if including itself.
  Therefore, the number of paths here is at most
  $\boldsymbol{\alpha^{2}|N_{2}|}$.

  \medskip

  \emph{Case 3: The path travels from $N_{3}$ to $N_{2}$ to $N_{3}$.} We
  can count these by looking at the vertex in $N_{2}$ first, and then
  picking 2 of its children in $N_{3}$. Thus, the number of such paths is
  at most $\boldsymbol{|N_2| \alpha^2}$.

  \medskip

  \emph{Case 4: The path travels from $N_{3}$ to $N_{2}$ to $N_{1}$.}
  Similarly to Case 2, the number of such paths is at most
  $\boldsymbol{\alpha |N_2| |N_1|}$.

  \medskip

  \emph{Case 5: The path travels from $N_{3}$ to $N_{2}$ to $N_{2}$.} By
  Lemma \ref{lem:N2-bdd-degree}, the number of edges inside $N_{2}$ is at
  most $|N_{2}||N_{1}|\frac{\alpha}{\alpha-\lfloor\alpha\rfloor}$.  Each
  such path consists of one of these edges, together with an edge to $N_3$
  from one of its two endpoints.  Therefore, the number of paths for us to
  count is at most
  $\boldsymbol{|N_{2}||N_{1}|\frac{2\alpha^{2}}{\alpha-\lfloor\alpha\rfloor}}$.

  \medskip

  \emph{Total:} summing over all cases, we have:
  \begin{align*}
    P
    &\leq
    |N_{3}\cup N_{4}|\alpha^{2}
    +2\alpha^{2}|N_{2}|
    +|N_{1}||N_{2}|\alpha
    +|N_{1}||N_{2}|\frac{2\alpha^{2}}{\alpha-\lfloor\alpha\rfloor}
    \\
    &<2\alpha^{2}n+|N_{1}|n\left(\alpha+\frac{2\alpha^{2}}{\alpha-\lfloor\alpha\rfloor}\right)
    \\
    &<|N_{1}|n \left(\frac{5\alpha^{2}}{\alpha-\lfloor\alpha\rfloor}\right) \,.
  \end{align*}
  But $P\geq|N_{3}\cup N_{4}|\frac{n}{\alpha}$ from above, so:
  \begin{align*}
    |N_{3}\cup N_{4}|\frac{n}{\alpha}
    &<
    |N_{1}|n \left(\frac{5\alpha^{2}}{\alpha-\lfloor\alpha\rfloor}\right) \\
    |N_{3}\cup N_{4}|
    &<
    |N_{1}| \cdot \frac{5\alpha^{3}}{\alpha-\lfloor\alpha\rfloor} \,.
  \end{align*}
\end{proof}

\begin{lemma}
  If every vertex has degree $>\sqrt{n \log n}$, then the graph is
  asymptotically socially optimal: the total social cost is at most
  $2n^{2}+\alpha n^{3/2} \sqrt{\log n}$.
  \label{lem:if-degrees-big}
\end{lemma}

\begin{proof}
 Suppose we have a Nash equilibrium where all vertices have degree greater
 than $\sqrt{n \log n}$.  We give a strategy for an arbitrary vertex to
 achieve an individual cost of at most $\alpha \sqrt{n \log n} +2n$, by
 changing only its own behavior.  Since this is a Nash equilibrium, we will
 then be able to conclude that every vertex must have had individual cost
 at most $\alpha \sqrt{n \log n} + 2n$, proving this claim.
 
 Specifically, we show that for any vertex $w$, the strategy ``undo all
 edges you're currently paying for, and connect to $\sqrt{n \log n}$
 vertices at random'' has a positive probability of bringing it within
 distance $\leq 2$ from every other vertex in the graph.  Indeed, if $w$
 does this, then for any other vertex $x$, 
 \begin{align*}
   &\ \pr{\text{$x$ is now distance $> 2$ from $w$}} \\
   \leq
   &\ \pr{\text{$w$ didn't choose any of $x$'s neighbors}} \\
   \leq
   &\ \left(1-\frac{\sqrt{n \log n}}{n}\right)^{\sqrt{n \log n}} 
   \leq
   e^{-\log n}
   =
   \frac{1}{n}
   \,.
 \end{align*}
 Since there are only $n-1$ other vertices $x \neq w$ to consider, a union
 bound shows that the probability of failure is at most $(n-1) \frac{1}{n}
 < 1$, and therefore there is a way for $w$ to attain an individual cost of
 at most $\alpha \sqrt{n \log n}+2n$, as desired.
\end{proof}

\begin{lemma}
  Even if there is a vertex of degree at most $\sqrt{n \log n}$, the graph
  is still asymptotically socially optimal: the total social cost is at
  most
  $2n^{2}+n^{3/2} \sqrt{\log n} \cdot \frac{290\alpha^{6}}{(\alpha-\lfloor\alpha\rfloor)^{2}}$.
  \label{lem:if-not}
\end{lemma}

\begin{proof}
  Let $v$ be a vertex of degree at most $\sqrt{n \log n}$, and construct
  the vertex partition $N_{1},N_{2},N_{3},N_{4}$. We already know $|N_{1}|$
  is at most $\sqrt{n \log n} = o(n)$, so by Lemma \ref{lem:N34<N1},
  $|N_{3} \cup N_{4}|$ is at most $\sqrt{n \log n} \cdot
  \frac{5\alpha^{3}}{\alpha-\lfloor\alpha\rfloor}=o(n)$.  By Lemma
  \ref{lem:N2-bdd-degree}, the total number of edges inside $N_{2}$ is at
  most $|N_{2}||N_{1}|\frac{\alpha}{\alpha-\lfloor\alpha\rfloor}\leq
  n^{3/2}\sqrt{\log n} \cdot
  \frac{\alpha}{\alpha-\lfloor\alpha\rfloor}=o(n^{2})$.  Also, the total
  number of edges not completely inside $N_{2}$ is at most
  $n\cdot(1+|N_{1}|+|N_{3}\cup N_{4}|)\leq n^{3/2} \sqrt{\log n} \cdot
  \frac{6\alpha^{3}}{\alpha-\lfloor\alpha\rfloor}=o(n^{2})$.  Therefore,
  the total number of edges is the whole graph is at most $n^{3/2}
  \sqrt{\log n} \cdot
  \frac{7\alpha^{3}}{\alpha-\lfloor\alpha\rfloor}=o(n^{2})$.

  Next, we calculate a bound on the total sum of distances in the graph.
  Using Lemma \ref{lem:N34<N1} on every vertex in the graph, and the fact
  that all distances are at most 4 (Corollary \ref{cor:diam-4}), we get:
  \begin{align*}
    &\ \text{[total sum of distances in the graph]} \\
    \leq &\ 
    2n^2 + 4 \text{[\# of distances in the graph that are 3 or 4]} \\
    = &\ 
    2n^2 + 4 \sum_{w}[\text{\# of vertices at distance 3 or 4 from $w$}] \\
    < &\ 
    2n^2 + 4 \sum_{w}\deg(w) \cdot \frac{5\alpha^{3}}{\alpha-\lfloor\alpha\rfloor}
    \,.
  \end{align*}
  The degree sum is precisely twice the total number of edges in the graph,
  a quantity which we just bounded above.  Putting everything together, the
  total sum of distances is at most:
  \begin{displaymath}
    2n^2 + 8 n^{3/2} \sqrt{\log n} 
    \cdot \frac{7\alpha^{3}}{\alpha-\lfloor\alpha\rfloor} 
    \cdot \frac{5\alpha^{3}}{\alpha-\lfloor\alpha\rfloor}
    =
    2n^2 + n^{3/2} \sqrt{\log n} \cdot 
    \frac{280\alpha^{6}}{(\alpha-\lfloor\alpha\rfloor)^{2}} \,.
  \end{displaymath}
  Adding $\alpha$ times the number of edges to compute the total social
  cost, we obtain the desired bound.
\end{proof}

Lemmas \ref{lem:if-degrees-big} and \ref{lem:if-not} cover complementary
cases, so we now conclude that the total social cost of every Nash
equilibrium is at most the bound obtained in Lemma \ref{lem:if-not}.  As
was observed by previous authors \cite{FaLuMaPaSh}, the social optimum for
$\alpha \geq 2$ is the star, achieving a social cost of at least $2n(n-1)$.
Dividing, we find that the price of anarchy is at most
\begin{displaymath}
  1 + \frac{150\alpha^{6}}{(\alpha-\lfloor\alpha\rfloor)^{2}}
  \sqrt{\frac{\log n}{n}} 
  =1+o(1) \,,
\end{displaymath}
proving the first part of Theorem \ref{thm:main}.

\section{Integral $\boldsymbol{\alpha}$}

There is one catch in our bound above. Namely, when $\alpha$ is only
slightly greater than an integer (e.g. 4.0001), the terms of the form
$\frac{\star}{\alpha-\lfloor\alpha\rfloor}$ all blow up, giving the final
$o(n^{2})$ terms for our bound a large constant factor. Even worse, when
$\alpha$ is an exact integer, the proof fails completely.  Perhaps
surprisingly, this is not an artifact of the proof.  In this section, we
construct a counterexample when $\alpha$ is an integer.  Let
$v_{1},v_{2},\ldots,v_{k}$ be a large clique with edges oriented
arbitrarily. In addition, each vertex $v_{i}$ in the clique also pays for
edges to $\alpha-1$ separate leaves
$l_{i:1},l_{i:2},\ldots,l_{i:\alpha-1}$.  This graph also appears in
\cite{AlEiEvMaRo}, but we provide a full analysis here for completeness.

\begin{lemma}
  In this graph, no single vertex has a better strategy than the one it's
  currently using.
  \label{lem:Z-eq}
\end{lemma}

\begin{proof}
  First, consider any leaf, say $l_{1:1}$.  This leaf is not currently
  paying for any edges, so its only option is to pay for some set of edges.
  Notice that purely choosing some set of edges to pay for, without being
  able to delete any edges, is an instance of convex optimization.
  Therefore, by convexity, if any vertex in any graph can improve its
  station purely by adding some set of edges $S$, then it can also do this
  by adding some single edge $s\in S$. By observation, the leaf can only
  break even by adding one edge, so it can only break even overall.

  Next, consider a members of the clique, say $v_{1}$. This vertex cannot
  delete its connections to its leaves, because that would disconnect the
  graph, making the distance component of its cost infinite.  If $v_{1}$
  remains neighbors with $v_{i}$ and also buys an edge to some leaf
  $v_{i:j}$, then this is suboptimal: the edge to $v_{i:j}$ costs $\alpha$
  but only gets $v$ closer to one vertex. If $v_{1}$ deletes its edge to
  $v_{i}$ but buys an edge to some leaf $v_{i:j}$, this is unnecessary:
  $v_{1}$ can move the edge from $v_{i:j}$ to $v_{i}$, switching its
  distances to those two vertices and not increasing the distance to any
  other vertex. Therefore, it is unnecessary for $v_{1}$ to consider
  strategies involving connecting to other vertices' leaves.

  Thus, similarly to the previous case, $v_{1}$ only needs to consider
  strategies involving purely deleting edges. Again, by convexity, this
  reduces to considering strategies involving deleting a single edge. But
  again, $v_{1}$ can only break even by deleting an edge, so it can only
  break even overall.
\end{proof}

Therefore, the graph is indeed a weak Nash equilibrium. Let $n$ be the
number of vertices in the graph.  The size of the clique is $k =
\frac{n}{\alpha}$, and so the cost of all of the edges is
\begin{displaymath}
  \alpha\left[\binom{k}{2} + (\alpha-1)k\right]
  =
  \alpha\left[ (1+o(1)) \frac{n^2}{2\alpha^2} + \frac{\alpha-1}{\alpha} n  \right]
  =
  (1+o(1)) \frac{n^2}{2\alpha} \,.
\end{displaymath}
Every clique vertex is distance 1 from the rest of the clique, as well as
its leaves, and distance 2 from every other vertex; therefore, each clique
vertex sees a distance sum of
\begin{displaymath}
  (1+o(1))n \left( 2 - \frac{1}{\alpha} \right) \,.
\end{displaymath}
Since there are $\frac{n}{\alpha}$ clique vertices, these contribute a
total of
\begin{displaymath}
  (1+o(1)) n^2 
  \left( \frac{2}{\alpha} - \frac{1}{\alpha^2} \right) \,.
\end{displaymath}
Every leaf vertex is distance 2 from almost all of the clique, and distance
3 from almost all of the leaves, and so it sees a distance sum of
\begin{displaymath}
  (1+o(1)) n \left( 3 - \frac{1}{\alpha} \right) \,.
\end{displaymath}
Since there are $n(1-\frac{1}{\alpha})$ leaves, these contribute a total
distance sum of
\begin{displaymath}
  (1+o(1)) n^2 \left( 3 - \frac{4}{\alpha} + \frac{3}{\alpha^2} \right) \,.
\end{displaymath}
Putting everything together, we find that this graph has a total social
cost of
\begin{displaymath}
  (1+o(1)) n^2 \left( 3 - \frac{3}{2\alpha} + \frac{2}{\alpha^2}  \right)
  \,,
\end{displaymath}
giving a price of anarchy at least $\frac{3}{2} - \frac{3}{4\alpha} +
\frac{1}{\alpha^2} +o(1)$, as claimed in the second part of Theorem
\ref{thm:main}.

\section{Concluding remarks}

It is interesting that the price of anarchy converges to 1 for non-integral
$\alpha>2$, but is bounded away from 1 for integer $\alpha\geq2$.
Our convergence rate is non-uniform in the sense that it slows down
substantially when $\alpha$ is slightly more than an integer. On
the other hand, when $\alpha$ is slightly less than an integer, the
convergence rate is still relatively rapid. It would be nice to prove
a uniform convergence rate for all non-integral $\alpha$.


\bibliographystyle{acm}
\bibliography{network-game}

\end{document}